\documentclass[10pt,oneside,english]{amsart}
\usepackage[T1]{fontenc}
\usepackage[latin9]{inputenc}
\usepackage{mathrsfs}
\usepackage{amsthm}
\usepackage{amstext}
\usepackage{amssymb}
\usepackage[all]{xy}

\makeatletter
\numberwithin{equation}{section}
\numberwithin{figure}{section}
\theoremstyle{plain}
\newtheorem{thm}{\protect\theoremname}
  \theoremstyle{plain}
  \newtheorem{prop}[thm]{\protect\propositionname}
  \theoremstyle{remark}
  
  \theoremstyle{plain}
  \newtheorem{lem}[thm]{\protect\lemmaname}
  \theoremstyle{plain}
  \newtheorem{cor}[thm]{\protect\corollaryname}
\newtheorem{mydef}{Definition}
\newcommand{\naive}{na\"{\i}ve}
\newtheorem{myax}{Axiom}
\newtheorem{myres}{Result}

\numberwithin{equation}{section}
\numberwithin{figure}{section}
\theoremstyle{plain}

\usepackage[english]{babel}
\usepackage{babel}

  \theoremstyle{plain}
  \theoremstyle{definition}
  \newtheorem{example}[thm]{\protect\examplename}

\usepackage{babel}
\providecommand{\corollaryname}{Corollary}
  \providecommand{\examplename}{Example}
  \providecommand{\lemmaname}{Lemma}
  \providecommand{\remarkname}{Remark}
\providecommand{\theoremname}{Theorem}
\providecommand{\propositionname}{Proposition}

\makeatother

\usepackage{babel}
  \providecommand{\corollaryname}{Corollary}
  \providecommand{\lemmaname}{Lemma}
  \providecommand{\remarkname}{Remark}
\providecommand{\theoremname}{Theorem}

\begin{document}
\global\long\def\qrrel#1#2#3{\hat{#1}\hat{#2}\left(#3\right)}

\global\long\def\qrabstract#1#2{\hat{#1}\left(#2\right)}

\global\long\def\usc#1{\mathscr{P}_{1}(#1)}
\global\long\def\sc#1{\mathscr{P}(#1)}
\global\long\def\Nn{\mathbb{N}}
\global\long\def\closure#1#2{*#1``#2}
\global\long\def\code#1{\ulcorner#1\urcorner}
\global\long\def\varusc#1{\iota``#1}
\global\long\def\tuple#1{\langle #1\rangle }
\global\long\def\imp{\Rightarrow}
\global\long\def\bic{\Leftrightarrow}
\global\long\def\cat#1{\mathbf{#1}}
\global\long\def\inj{\rightarrowtail}
\global\long\def\restric{\restriction}
\global\long\def\fst{\mathtt{fst}}
 \global\long\def\snd{\mathtt{snd}}

\title{Category Theory with Stratified Set Theory}

\author{Thomas Forster, Adam Lewicki, Alice Vidrine}

\begin{abstract}
	
	This paper examines the category theory of stratified set theory (NF and KF). We work out the properties of the relevant categories of sets, and introduce a functorial analogue to Specker's T-operation. Such a
        development leads one to consider the appropriate notion of ``elementary topos'' for stratified set theories. In addition to considering the categorical properties of a generic model of NF set theory, we identify a stratified Yoneda Lemma and show NF encodes itself as a full internal subcategory. Finally, our desire to examine NF in the context of category theory motivates a more precise
        examination of strongly cantorian as an appropriate notion of smallness, replacing it with the notion of \textit{fibrewise} strongly cantorian. In the absence of Choice, we introduce a new axiom (SCU) to NF, and examine some properties of NF + SCU.
	
\end{abstract}

\maketitle

\section{Introduction}

Category Theory and Set Theory form a natural duality. The study of this duality, in the
context of classical ZF-like set theories, is well developed. In set theories with stratified comprehension/separation axioms, however, there has been no significant investigation. We initiate this project by examining the
properties of $\cat{N}$, the category of NF sets, as well as touching upon the internal and fibred category theory of NF.

As a foundational set theory for category theory, NF has its appeal. Size issues in ZF(C) require one to work with a hierarchy of classes and super-classes, or to accept solutions such as Grothendieck Universes, often requiring cardinal assumptions. NF's capacity to internalize traditionally ``large'' categories provides an elegant solution. There are, however, a number of desirable foundational properties which NF does not possess. Rather than size restrictions, stratified comprehension places restrictions on syntax, which results in neither a cartesian closed category of sets, nor a suitably natural form of the Yoneda Lemma. The failure of cartesian closure was shown in \cite{mclarty}, and \cite{fef} at least observes the typing issue involved in stating the Yoneda
Lemma, in its classical form. By implementing Specker's T-operation, as an endofunctor, we provide stratified analogues of
both. 

Lawvere has referred to toposes as ``variable'' set theories,\footnote{A notion inspired by the use of Grothendieck's eponymous use of toposes of sheaves, in the study of algebraic geometry.} with \textbf{Set} providing the canonical example of an elementary topos - the \textit{stationary} topos of sheaves over the trivial space. We should not expect $\cat{N}$ to form a topos, as the notion was developed in the context of unstratified set theory. We can, however, investigate the fundamental \textit{categorical} properties of a model of sets for a stratified theory. We provide a preliminary definition of an \textit{NF-Topos}, motivated by KF and NF.\footnote{In acknowledgment of their colleague, professor, and friend's fundamental contribution to Topos Theory, the authors refer to such an object as an \textbf{SPE} - \textit{stratified pseudo elephant}.} 

Related to the T-operation, strongly cantorian sets have long been considered the ``small'' sets in NF folklore. We show the category of
strongly cantorian sets does, in fact, form a topos. Beyond this, we initiate the use of Algebraic Set Theory, in the context of NF, to examine the appropriate \textit{categorical} notion of smallness. As is typical of category theory, the smallness condition is one placed on maps. Our ``small'' maps are \textit{fibrewise} strongly cantorian. While it is a theorem of NFU + Choice that these form a class of small maps, in the sense of AST, working in NF requires a new axiom, \textbf{SCU}. 

As with any paper that attempts to bridge the gap between largely disjoint areas of study, the reader is likely to encounter the twin frustrations of results that are too elementary and ones that require background not covered within the paper. We make every effort to alleviate these frustrations. In the end, however, length dictates that we reference out much of the required knowledge of category theory. The relative lack of introductory material on NF leads to its larger role within our exposition.

The paper is organized as follows: We provide a brief overview of stratified set theories, and work out the basic properties of the
categories of KF and NF sets. This motivates a speculative definition and study of a stratified topos, which we call an
\textit{SPE}. The abstract definition is motivated by our study of $\cat{N}$, and we proceed by proving their manifestations set theoretically in NF. The following section focuses on the internal category theory of $\cat{N}$. In the final section, we examine a modified notions of smallness in NF, which leads to desirable properties for both category and set theory.

\section{Background}
\subsection{New Foundations Set Theory}
There are two textbooks on New Foundations\footnote{A third book, due to Randall Holmes, introducing NFU (NF + Urelemente) is also an excellent resource}: \cite{forster2} is the only modern textbook, containing both introductory material and an advanced survey; \cite{lfm} is quite old, but remains an excellent, methodical introduction to the theory.

Before introducing the axioms of NF set theory, we must first define an
appropriate notion of \emph{stratification}. A \emph{stratification} of a
formula, in the language of set theory, is a function $\sigma$ from the (not necessarily free) variables of the
formula to the natural numbers satisfying:

i) Every occurrence of a given variable $v_i$ has the same value under $\sigma$

ii) If $v_i$ and $v_j$ appear in the context $v_i = v_j$, $\sigma (v_i) = \sigma
(v_j)$

iii) If $v_i$ and $v_j$ appear in the context $v_i \in v_j$, $\sigma (v_i) + 1 =
\sigma (v_j)$.

The \emph{stratifiable} formulae of (untyped) set theory are precisely the well-formed formulae of Russell's Simple Theory of Types (TST), dropping the typing indices.

NF can be axiomatised as extensionality, plus stratifiable
instances of the comprehension scheme $$ \forall \overrightarrow{x} \exists y
\forall z (z \in y \leftrightarrow \Phi(\overrightarrow{x},z))$$ with $y$ not
occurring free in $\Phi$. The mantra of NF could be: restrict by complexity, rather than size. While NF admits the largest set, $V = \{x | x = x \}$, it avoids Russell's paradox, as the formula expressing self-containment, $ x \notin x$, is unstratified. The restriction to stratified comprehension also allows NF to admit large sets like $On$, the set of all ordinals, without forming the paradox of Burali-Forti. 

Implementation of ordered pairs is a benign component of unstratified theories. For stratified theory, however, the type $\langle x,y \rangle$ receives, relative to `$x$' and `$y$' in a given formula is critical. By refuting Choice in NF, Specker proved the axiom of infinity in NF.\cite{specker} This allowed Quine to implement a type-level, surjective pairing function, which we utilize extensively. Regardless of implementation, however, one is not able to define evaluation (currying) as a function of NF.

\subsubsection{Results of Interest in NF}
Specker's proof of the failure of Choice and, as a direct corollary, the existence of an infinite set is probably the most well-known result of NF. Although his ``Dualitat'' is possibly the most beautiful, proving equiconsistency between NF and TST + Ambiguity \cite{specker2}.

Another side effect of working in NF is the absence of Replacement. This requires one to consider functions which are ``setlike.''\footnote{Permutations (possibly external) of the universe are an important component of the model theory of NF(U) and have a deep semantic connection to stratified formulae. The reader is referred to \cite{forster2}. }  

\begin{mydef}
	For some function $f$, we define $f``$ as $f$ ``acting one level down," thus $f``X = \{f(x)| x \in X \}$. For $f(x)$ we sometimes write $f`x$, when it is helpful to do so. $f``$ is often written $j(f)$ (for $j$ump), allowing one to generalize to $j^n(f)$, the function obtained by applying $f$ $n$ levels down.
\end{mydef}

If, despite the lack of replacement, $f``x$ always
exists, for some class function $f$, we say that $f$ is $1$-\textbf{setlike}. We say $f$ is $n$-\textbf{setlike} if $j(n)(f)`x$ exists for all $x$. If this holds for all $n$, $f$ is said to be \textbf{setlike}.

\begin{mydef}
	To denote the singleton of a set $x$, $\{x\}$, we write $\iota`x$, with $\iota``x$ denoting the set of singletons of elements of $x$, $\{\{y\}| y \in x \}$.
\end{mydef}

Regardless of implementation, a minimal requirement of ordered pairs is that $\langle x,y \rangle$ is such that $x$ and $y$ are type-level. If this is not the case, we are not able to compose relations. As a result, $\iota$ is not a function of NF, as $x$ and $\{x\}$ cannot receive the same type in a stratified formula.

\begin{mydef}
	We say that a set $x$ is \textbf{cantorian} if $|x| = | \iota `` x |$. 
\end{mydef}

\begin{mydef}
	A set is said to be \textbf{strongly cantorian} if the graph of $\iota
	\restric x$ is a set (hence, witnessing $Can(x)$).
\end{mydef}

The use of ``cantorian'' is motivated by another key result of NF, the stratified version of Cantor's theorem. While the diagonalization property is unstratified, we can invoke the external(!) bijection between $\iota``x$ and $x$ to state a stratified version, and modify Cantor's Theorem to prove: $\iota``x$ is strictly smaller than $\mathcal{P}X$. An immediate corollary of this result is the proof that the set of
singletons of elements of $V$ is strictly smaller than $V$ (as $V$ is its own powerset). Corollary to this is a key property of $\cat{N}$, the category of NF sets: the failure of the global sections functor $\cat{N}(1,-)$ to be essentially surjective.

$\cat{N}(1,-)$ is isomorphic to $\iota``$, viewed as functor. We refer to this as the $T-functor$, acknowledging its connection to Specker's T-operation.

\begin{mydef}
The \textbf{T-functor} is defined by its action on objects, $ x \mapsto \iota``x$, and its action on maps, $f:A \to B \mapsto Tf:TA \to TB$ defined as $\langle \{a\},\{b\}\rangle \in Tf$ where $\langle a,b \rangle$ is in $f$.
\end{mydef}

As with the proof of NF's version of Cantor's Theorem, the external bijection between $\iota``x$ and $x$ allows one to ``type-raise'' $x$ in a given formula that would otherwise be unstratified.\footnote{In the category theory of stratified sets, this manifests as the ubiquity of \emph{T-relative adjunctions}.} In particular, if $x$ is the same size as a set of singletons, $\iota``\bigcup x = x$. We introduce two axioms, which aim to exploit this.

\begin{mydef}
\textbf{IO} is the principle that every set is the same size as a set of singletons.
\end{mydef} 

While it clearly fails in full NF, IO is much weaker than the claim that every set is (strongly) cantorian. We later prove that the stronger assumption, added to the axioms of KF, yields Mac Lane set theory.

\begin{mydef}
\textbf{CE} is the principle that a family of pairwise disjoint sets is the same size as a set of singletons.
\end{mydef}

The CE principle allows us to deal with typing issues that arise in the construction of coequalizers. An immediate consequence is that any partition is the same size as a set of singletons.

\subsubsection{KF}

Kaye-Forster set theory is another stratified set theory examined in this
paper, and developed in \cite{KF}. A category theorist should be familiar with Mac Lane set theory (Zermelo
with $\Delta_0$-separation) as a natural set theoretic language for toposes. KF
is simply Mac Lane set theory with stratified $\Delta_0$-separation. KF is an interesting set
theory in which to work, for us, as both ZF and NF are extensions of it. NF is, in fact, simply KF + $\exists x \forall y (y \in x)$. 

\subsection{Category Theory}

The reader should be able to reference any introductory book on category theory, for questions related to its use in this paper. Where material is more obscure, we introduce in line with its use. Before moving to section 3, it is necessary to define a generalisation of adjoint functors, due to Ulmer.\cite{ulmer} A relative adjunction involves three categories and three functors,
\[
\begin{array}{c}
J:\mathscr E\to\mathscr D; F:\mathscr E\to\mathscr C; G:\mathscr C\to\mathscr D
\end{array}
\]

such that one has a correspondence of either of the following forms:
\[
\begin{array}{cc}
\dfrac{FA\to B}{JA\to GB} & \textrm{ written }F\;_J\negthickspace\dashv G\\
\,\\
\dfrac{B\to FA}{GB\to JA} & \textrm{ written }G\dashv_{J}F
\end{array}.
\]
In the former case we say that $F$ is a $J$\emph{-left adjoint to $G$}; in the
latter, that $F$ is $J$\emph{-right adjoint to $G$}. Clearly where
$J=id_{\mathscr E}=id_{\mathscr D}$, we recover an ordinary adjunction. It is
important to note that while adjunctions have both a unit ($\eta:id_{\mathscr
	D}\to GF$) and a counit ($\varepsilon:FG\to id_{\mathscr C}$), relative
adjunctions will generally have only one or the other (in this case $\zeta:J\to
GF$ or $\theta:GF\to J$).

\section{$\mathbf{N}$: The Category of NF Sets as an ``Almost Topos''}

In this section we provide a framework for studying $\cat{N}$ as a category of sets. Much of what we find (T-relative adjunctions, in place of standard adjunctions) that seems pathological, will be made to seem more natural in the following section on the internal category theory of NF. In the more general setting of an SPE, defined below, we simply require $T$ to be an endofunctor satisfying certain properties. In the remainder of this section, we prove the instance of $T$ in $\cat{N}$, defined earlier, satisfies these more abstract condition. We begin with the most basic categorical properties of stratified set theories (true in KF, as well as NF), and proceed to properties of greater complexity. We provide greater detail on relative adjunctions in line with our proof of pseudo cartesian closure in $\cat{N}$.

We define an SPE to generalize both a topos and the category of NF sets; in
particular, the characterization will make no appeal to the existence of a
``universe'' object, and will specialize to an NF-like category when one exists.

An SPE is a category $\mathscr{C}$ such that:
\begin{enumerate}
\item $\mathscr{C}$ is a regular category with finite coproducts and a subobject
  classifier,%
 
\item there is a full embedding $T:\mathscr{C}\to\mathscr{C}$ which creates
  finite limits%
  \footnote{We use the notion of creation of limits where for a diagram $F:\cat
    I\to\mathscr{C}$, a limit of $F$ exists whenever there is a limit of $T\circ
    F$, and $T$ preserves and reflects limits. This is somewhat weaker than the
    notion in \cite{Mac}, in that it doesn't require that all limiting cones lift.},
\item there is a bifunctor $\imp:\mathscr{C}^{op}\times\mathscr{C}\to\mathscr{C}$
with the following properties:

\begin{enumerate}
\item there is a natural isomorphism $i_{A}:TA\cong1\imp A$,
\item there is an extranatural transformation $\alpha_{A}:1\to A\imp A$,
\item there is a transformation $\beta_{A,B,X}:TA\imp TB\to(X\imp A)\imp(X\imp B)$
natural in $A,B$, extranatural in $X$,
\item there is a natural isomorphism $e_{A,B}:(TA\imp TB)\cong T(A\imp B)$,
\item and certain identities between these transformations hold%
  \footnote{We will not prove all of the extranaturality conditions for exponentiation, which are adjustments of the axioms for a closed category. It should become clear that it is an easy exercise  for a reader interested in working with stratification.%
  }.
\end{enumerate}
\item For all $f:A\to B$ there is a functor
  $\tilde{\Pi}_{f}:\mathscr{C}/A\to\mathscr{C}/TB$ for which
  $f^{*}\;_{T_{B}}\negthickspace\negmedspace\dashv\tilde{\Pi}_{f}$ (where
  $T_{B}:\mathscr{C}/B\to\mathscr{C}/TB$ is the obvious functor induced by $T$).
\item The functor $\mathsf{Sub}(TA\times-)$ is representable for any $A$%
  \footnote{Nathan Bowler should be acknowledged for pointing out this feature of NF to the authors.}
\end{enumerate}

\subsection{Basic Properties of Categories of Stratified Sets}
\subsubsection{Pairs, Finite Completeness, Finite Coproducts, and the Subobject Classifier}

Certain properties of the categories of KF and NF sets are independent of the implementation, but many are not. As stated above, the minimal requirement for a reasonable pairing function is that `$x$' and `$y$' receive the same type in $\langle x,y \rangle$. In full NF, we can implement Quine ordered pairs, where $\langle x,y \rangle$ receives the same type as both `$x$' and `$y$'.\cite{quine} In KF, however, such an implementation requires the existence of a Dedekind-infinite set.

Even in KF, one can prove certain straightforward properties of the category of sets. These are summarized in the following theorem:

\begin{thm}\
	
	\begin{enumerate}
		\item The category $\cat{K}$ of KF sets has an initial and terminal object, all equalizers, and a subobject classifier.
		
		\item If $\cat{K}$ is a model of KF + Inf, $\cat{K}$ is finitely complete and has finite coproducts.
	\end{enumerate}
	
\end{thm}
\begin{proof}
	The existence of an initial and terminal object is trivial (the empty set and any singleton set will do). For any two functions $f$ and $g$, the graphs of which are sets in KF, $\{x| f(x)=g(x)\}$ is a stratified set abstract, of the same type as the domain and codomain of $f$ and $g$. The subobject classifier is just the $2 = \{\bot,\top\}$ with the specified inclusion of $\{\top\}$ as \emph{true}.
	
	The second part requires type-level ordered pairs, to prove the existence of the sets defining the graphs of the projection ($\pi_A: A\times B \to A$; $\langle a,b \rangle \mapsto a$) and inclusion ($v_A:A \to A \sqcup B$) functions. We prove that the existence of a Dedekind-infinite set suffices for Quine Pairs in KF, as the following lemma.
\end{proof}

\begin{lem}
Let $T$ be a theory extending KF.  If $T$ proves the existence of a
Dedekind-infinite set then there is a flat pairing function
definable with the Dedekind-infinite set as a parameter.

This pairing function supports the existence not only of
$\tuple{x,y}$ for all $x$ and $y$ but also $x \times y$, $x \to y$
and inverses (locally) to all these constructs.

\end{lem}
\begin{proof}

	Suppose we have a Dedekind-infinite set, $X$ with an injection $f:X
	\inj X$ which is not surjective.  Then we have an implementation of
	$\Nn$ as follows.  Let $x \in X$ be anything not in the range of $f$
	and consider $\bigcap\{Y \subseteq X| x \in Y \wedge f``Y \subseteq
	Y\}$.  This object serves as $\Nn$, $x$ serves as $0$ and $f$ is
	successor.
	
	From here one can simply define the usual Quine pairing functions, as in NF. See \cite{forster} for the details.
	
\end{proof}

The axiom of infinity is a theorem of NF, and the implementation of a natural numbers object implies the existence of type-level pairs, as defined by Quine. In similar theories, such as NFU, the two have equivalent consistency strength. The situation in KF is more complex, as we have separation only for $\Delta_0$
formul{\ae}.

If we assume IO, then finite products and coproducts exist ``locally''. That is to say, if $A$ and $B$ are subsets of some large set $U$ we can form an
object that behaves like the product of $A$ and $B$, in the sense that the
{[}graphs of the{]} two projection functions are sets. This we do as follows. By
two applications of IO there will be a map $f$, defined on $U$, s.t. $(\forall
x\in U)(\exists y)(f(x)=\{\{y\}\})$.  The product of $A$ and $B$ (local to
$U$) will now be the set $\{\tuple{\bigcup^{2}(f(a)),\bigcup^{2}(f(b))}:a\in A\
\wedge\ b\in B\}$, where the pairs are Wiener-Kuratowski. This last fact ensures
that the two functions $\tuple{\bigcup^{2}(f(a)),\bigcup^{2}(f(b))}\mapsto a$ and
$\tuple{\bigcup^{2}(f(a)),\bigcup^{2}(f(b))}\mapsto b$ are defined by stratifiable
set abstracts.  Coproducts (disjoint unions) yield the same treatment.

\subsubsection{Coequalisers and Regularity}

Coequalisers are slightly more complicated.\label{sub:Coeq} In a standard set theory, the coequalizer $h:B \to C$ of arrows $f$ and $g:A \to B$, is formed as the quotient of $C$ by $\simeq$, the
$\subseteq$-least equivalence relation extending $$\{\tuple{b,b'}|(\exists a\in
A)(b=f(a)\wedge b'=g(a))\},$$ by taking $C$ to be $B/\simeq$, and $h$ to be
$\lambda_x(\imath y)(x\in y)$. To see why this strategy will not work in NF, we
need the following standard fact, provable in KF.
\begin{lem}
  \label{PTJ} For every set $B$ and any partition of $B$ we can find a set $A$
  and two maps $f,g:A\to B$ such that the partition is the set of equivalence
  classes of members of $B$ under the $\subseteq$-least equivalence relation
  extending $\{\tuple{b,b'}|(\exists a\in A)(b=f(a)\wedge b'=g(a))\}$.
\end{lem}
\begin{proof}
  Given a partition $B'$ of $B$, form the (equivalence) relation $\{\tuple{x,y}|
  \exists b\in B'(x\in b\,\wedge \, y\in b)\}$ and take this to be our
  $A$. The pair of morphisms needed will just be the projections $A\to B$.
\end{proof} 
\begin{lem}
  NF refutes that every pair of projections in the above situation can be
  coequalised by the quotient set equipped with the morphism $\lambda_x(\imath
  y)(x\in y)$.\end{lem}
\begin{proof}
  Let $B$ be the set of all wellorderings and $A$ be
$$\left\{\tuple{x,y}|x,y\in B\,.\, x\textrm{ is order-isomorphic to
  }y\right\}.$$ Then the quotient of $A$'s projections is just $NO$, the ordinal
  numbers. If $c:=\lambda x.(\imath y)(x\in y):B\to NO$ is a set, then then so
  is the function $j^2(c):NO\to\varusc{NO}$. But because each ordinal would be
  the sole value of its members under $c$, this means that $j^2(c)$ is simply
  the singleton function. But, by arguments the reader can find in \cite{lfm},
  such a singleton function would allow one to prove the Burali-Forti paradox;
  we present a similar argument in more detail on page \pageref{bfort}. Since
  the defining condition for such a function is unstratifiable, we have
  disproved its existence.\end{proof}
\begin{cor}
  If $A$ is a family of non-empty, disjoint sets, then the existence of a
  membership morphism $\bigcup A\to A$ implies that $A$ is strongly
  cantorian.\end{cor}
\begin{thm}
(KF)

(i) ``Every coequaliser diagram can be completed'' is equivalent to

(ii) ``Every set of pairwise disjoint sets is the same size as a set of
  singletons''
\end{thm}
\begin{proof}

  (ii) $\to$ (i). Let $B$ be the target set in a coequaliser diagram, and let
  $\simeq$ stand for the $\subseteq$-least equivalence relation extending
  $\{\tuple{b,b'}|(\exists a\in A)(b=f(a)\wedge b'=g(a))\}$.  By assumption
  there is a bijection $H$ from the set $\{[b]_{\simeq}|b\in B\}$ of equivalence
  classes to some set $\iota``C$ of singletons. Now
  define $h(b)$ to be $\bigcup H([b]_{\simeq})$.

\smallskip{}

(i) $\to$ (ii). Let $\Pi$ be a set of pairwise disjoint sets, and let
$B=\bigcup\Pi$. By lemma \ref{PTJ} we can find a set $A$ and maps $f,g:A\to B$
such that $\Pi$ is the induced partition. What we want is a set $C$ such that
there exists one element for every equivalence class of $\Pi$ and a function $f$
mapping each element of $B$ to the corresponding element. Note that this doesn't
necessarily mean the correspondence between $\Pi$ and $C$ is a \emph{set}! Such
a $C$ and such a map would constitute a coequaliser, and $j^2(f)(\Pi)$ is a set of
singletons in 1-1 correspondence with $\Pi$.

\end{proof}

It is not evident, whether or not NF proves the existence of arbitrary coequalizers; the result would immediately imply $\cat{N}$ is finitely cocomplete. In $\cat{N}$, cocompleteness is equivalent to the property CE, defined above. 

Clearly even a weakened version of Choice (for partitions) would suffice for one to pick representatives of each equivalence class, allowing us to form the graph which defines the coequalizer:
\begin{cor}
NFU + Choice has coequalizers.
\end{cor}

We can prove NF has certain coequalisers and, in a similar manner, that all epimorphisms in $\cat{N}$ are regular. 

\begin{lem}
	NF has coequalizers of kernel pairs.
\end{lem}
\begin{proof}
The basic construction is to show that the collection of fibres, of any function, is the same size as the set of singletons of elements of the image. To a student of NF, this allows us to subvert typing issues. In the language of category theory, we use the isomorphism between the fibres and an object in the image $T$, which is shown to be full and faithful.
\end{proof}

\begin{cor}
	$\cat{N}$, the category of NF sets is a regular category.
\end{cor}

Additionally, a parallel pair of arrows $Tf,Tg:TA\rightrightarrows
TB$ will always have a coequaliser. This is the first example of the ubiquity of \emph{relative} adjoints in the category theory of stratified sets.

\begin{example}
	In the classical case, a category $\mathcal{C}$ has coequalizers precisely when there is a left adjoint to the functor $\Delta$ sending an object $X \in \mathcal{C}$ to the constant functor $\Delta_X:(\cdot\rightrightarrows\cdot)\to \mathcal{C}$. Said left adjoint sends a pair of arrows to their coequalizer.
	
	In $\cat{N}$, this corresponds precisely to the relative adjunction $\overline{G} \;_{T^{\cdot \rightrightarrows \cdot}}\negthickspace \dashv \Delta$, where $\overline{G}$ is the $T^{\cdot \rightrightarrows \cdot}$-left adjoint of $\Delta$, where $T^{\cdot \rightrightarrows \cdot}: \cat{N}^{\cdot \rightrightarrows \cdot} \to \cat{N}^{\cdot \rightrightarrows \cdot}$ is the functor given by post-composition with $T$:
	$$ \cat{N}^1(coeq(Tf,Tg),C) \cong \cat{N}^{\cdot \rightrightarrows \cdot}((Tf,Tg),\Delta_C)$$
\end{example}

\subsection{Pseudo-Cartesian Closure and Local Pseudo-Cartesian Closure}

It is well known that $\cat{N}$ is not cartesian closed. In KF, however, the implication is not inconsistency, but something category theorists know quite well.

\begin{thm} KF + ``The category of sets is cartesian closed'' is equivalent to Mac Lane set theory. \end{thm}

\begin{proof}

Mac = KF + ``every set is strongly cantorian''. We prove that the latter is equivalent to cartesian closure.

If the graph of \verb#curry# is, locally, a set, so too is
the graph (call it $f_{1}$) of the function that for each $x\in A$ sends
$(\{\emptyset\}\times\{\emptyset\})\to x$ to $\{\emptyset\}\to(\{\emptyset\}\to
x)$.  Now $\{\emptyset\}\to x$ is one type higher than $x$ so---by stratified
comprehension---the graph (call it $f_{2}$) of the function sending $\{x\}$ to
$(\{\emptyset\}\times\{\emptyset\})\to x$ is a set. By the same token
$\{\emptyset\}\to(\{\emptyset\}\to x)$ is two types higher than $x$, and---by
stratified comprehension again---the graph (call it $f_{3}$) of the function
sending $\{\emptyset\}\to(\{\emptyset\}\to x)$ to $\{\{x\}\}$ is also a set.

Then the composition $f_{3}\circ f_{1}\circ f_{2}$ sends $\{x\}$ to
$\{\{x\}\}$. Thus $f_{3}\circ f_{1}\circ f_{2}$ is $\iota\restriction\iota``A$,
which is to say that $A$ is strongly cantorian. But $A$ was arbitrary.

\end{proof}

KF + Cartesian Closure gives us an example of an SPE where $T$ is naturally isomorphic to the identity functor. In NF, the $T$ functor described earlier (naturally isomorphic to $\cat{N}(1,-)$) is sufficient to describe the aspect of an SPE we call \emph{pseudo} cartesian closure. 

\begin{mydef}
	$A\imp B$ denotes the set of functions between sets A and B. If $A\imp B$ satisfies the conditions of an \emph{exponential} we denote it $B^A$.\footnote{In this sense, NF gives us a stark example of a situation where an object satisfying a universal property may be very different from the object which is the predicate-in-extension of the intuitive set-theoretic definition.}
\end{mydef}

For any two sets $A,B$ of NF, the set $A\imp B$ exists; cartesian closure fails because the evaluation
arrow $ev_{A,B}:A\times(A\imp B)\to B$, ``$ev(\tuple{x,f})=f(x)$,'' cannot be stratified. We can instead choose a proxy
definition which tells us indirectly what we need to know about $\hom(A,B)$ by
observing that the expression ``$\tuple{\{x\},f}=\{f(x)\}$'' is stratifiable;
thus if $f:A\to B$, we do always have an evaluation arrow, $ev'_{A,B}$ with the
expected property from $TA\times(A\imp B)$ to $TB$.\footnote{ Nathan Bowler and Morgan Thomas have formulated this property independently. The novelty of this paper is the proof that it forms a relative adjunction, and the development of a localized version.}
\begin{prop}
  for all $A,B$, there is a $(TA\times-)$-co-universal arrow $\tuple{A\imp
    B,ev'_{A,B}}$ to $TB$; equivalently, we have the relative adjunction
  $(TA\times-)\dashv_{T}(A\imp-)$ with co-unit $ev'$.
  \end{prop}
  
\begin{proof}\footnote{For economy of language, we use $\lambda$ notation in the following proofs, putting ``$\lambda x.\phi$'' where one might write ``$x\mapsto\phi$''.}
  We want a natural isomorphism $\theta:\hom(TA\times C,TB)\cong\hom(C,B^{A})$.
  We let $\theta$ take $f:TA\times C\to TB$ to the function
\[
\overline{f}:=\lambda c.\lambda a.\left(\bigcup f(\{a\},c)\right).
\]
To see naturality, observe that given a function $h:C\to D$ and a $g:TA\times
D\to TB$ we can form the function $\tuple{\{a\},c}\mapsto g(\{a\},h(c))$ formed
by the composite $g\circ(id_{TA}\times h)$. We then take
$\theta_{C}(g\circ(id_{TA}\times h))$, or
\[
\lambda c.\lambda a.\left(\bigcup g(\{a\},h(c))\right).
\]
 If we instead take $g$ to $\theta_{D}(g)=\overline{g}:D\to B^{A}$, we can then
 take the composite $\overline{g}\circ h:C\to B^{A}$, which is again
\[
\lambda c.\lambda a.\left(\bigcup g(\{a\},h(c))\right),
\]
establishing naturality of $\theta$. For the inverse we have $\theta^{-1}$ take
$m:C\to B^{A}$ to
\[
\underline{m}:=\lambda{\tuple{\{a\},c}}.\left(m(c)(a)\right).
\]
Given $n:D\to B^{A}$, we can form the composite $n\circ h$ and apply
$\theta_{C}^{-1}$ to obtain
\[
\underline{n\circ h}=\lambda{\tuple{\{a\},c}}.\left(n\circ h(c)(a)\right).
\]
Alternatively, we can apply $\theta_{D}^{-1}$ to $n$ and get
$\underline{n}:TA\times D\to TB$.  Taking the composite
$\underline{n}\circ(id_{TA}\times h)$, we get exactly $\underline{n\circ h}$,
establishing that $\theta^{-1}$ is also natural. To see that these are in fact
inverse to one another, consider
\[
\theta_{D}(\underline{m})=\overline{(\underline{m})}=\lambda d.\lambda a.\left(\bigcup(\lambda{\tuple{\{a\},d}}.\left(m(d)(a)\right)(\{a\},d))\right)
\]
which reduces to $\lambda d.\lambda a.\left(m(d)(a)\right)$. Since
$\lambda d.\lambda a.(m(d)(a))(d)(a)=m(d)(a)$ and functions in NF are
extensional, we have that this has given us back $m$ again.  This gives us that
$\theta\circ\theta^{-1}=id_{\hom(-,B^{A})}$. To show
$\theta^{-1}\circ\theta=id_{\hom(TA\times-,TB)}$ we
reduce
\[
\lambda{\tuple{\{a\},d}}.\left(\lambda d.\lambda a.\left(\bigcup
g(\{a\},d)\right)(d)(a)\right),
\]
which is readily seen just to be $g:TA\times D\to TB$. We omit the proof of naturality in $B$ to save space - it is easier to see than naturality
in $C$.
\end{proof}
The asymmetry of relative adjunctions results in second relative adjoint, which captures a qualified form of the unit. 

\begin{prop}
	There exists a relative adjunction $(A\times-)_{\;T}\negthickspace\dashv(A\imp-)$.
\end{prop}

\begin{proof}
	
From a function $f:TC\to A\imp B$ one can form
a function $f':\tuple{a,c}\mapsto f(\{c\})(a)$.  If we let
\[
k_{C}(\{c\}):=\lambda a.(\tuple{a,c}):A\imp(A\times C),
\]
then $k$ is the unit of the relative adjunction $(A\times-)_{\;
  T}\negthickspace\dashv(A\imp-)$, returning $f$ as $A\imp f'\circ k$
(where $A\imp f':A\imp(A\times C)\longrightarrow A\imp B$ is just
postcomposition by $f'$). The arguments to establish this follow the lines of those above. 
\end{proof}
The presence of these relative
adjunctions proves:

\begin{cor}
	$\cat{N}$, the category of NF Sets, is pseudo-cartesian closed.
\end{cor}

The connection to cartesian closure and, in some sense, the symmetry of these relative adjunctions is provided by the following result.

\begin{prop} If $T$ preserves exponentials (in the sense of $T(A\imp B)$
being isomorphic to $TA\imp TB$ in a natural way), is full and faithful, and
creates limits, then each implies the other.
\end{prop}
\begin{proof}
 Let $f\imp g:A\imp B\to A'\imp B'$, for $f:A'\to
A$ and $g:B\to B'$, be the map that takes $h\in A\imp B$ to $g\circ h\circ f\in
A'\imp B'$. Then $T(f\imp g)$ takes $\{h\}\to\{g\circ h\circ f\}$ while $Tf\imp
Tg$ takes $Th$ to $Tg\circ Th\circ Tf$. Since ``$Th=\{h\}$'' is stratifiable,
the desired natural transformation exists and can be shown to be an
isomorphism. As $T(A\imp B) \cong TA\imp TB$, one can show that $$\hom(C,A\imp B)\cong\hom(TC,T(A\imp B)) \cong\hom(TA\times C,TB)$$ where the first isomorphism comes from the fullness and faithfulness of $T$. The
only relative adjunction used in this string of isomorphisms is $(A\times-)_{\;T}\negthickspace\dashv(A\imp-)$. A similar one using only the $(TA\times-)\dashv_{T}(A\imp-)$ relative adjunction is even easier, since $$\hom(A\times C,B)\cong\hom(T(A\times C),TB)\cong\hom(TA\times TC,TB)\cong\hom(TC,A\imp B).$$
\end{proof}

\subsubsection{Connecting Subobjects and Powerobjects, in NF}
From pseudo-cartesian closure, we can obtain a pseudo-powerobject.\footnote{This is, effectively, Nathan Bowler's observation on representability of $Sub(TA \times -)$.} We wish to make this situation precise, further strengthening the link between an SPE and an elementary topos. 

\begin{mydef}
	We give a precise definition of \emph{pseudo-powerobjects}, satisfying the following universal property: for a subobject of $R \hookrightarrow TA \times B$, there is a unique morphism $\hat{r}: B \to PA$ such that the following diagram commutes, where $r$ is the characteristic morphism of $R$:

$$ \xymatrix{B \ar[d]^{\hat{r}} & TA \times B \ar[r]^<<<<{r} \ar[d]^{1_{TA} \times \hat{r}} & 2 \ar@{=}[d] \\ 
	PA & TA \times PA \ar[r]_<<<<{\in_A} & 2} $$ 

We refer to $\hat{r}$ as the \emph{P-transpose of $r$}.
\end{mydef}

The following lemmas hold for an arbitrary SPE, as well as the specific case of $\cat{N}$.

\begin{lem}\

$(TA\times-)\dashv_{T}(A\imp-)$ and the existence of a subobject classifier is equivalent to $Sub(TA \times -)$ being representable.
\end{lem}
\begin{proof}
We simply mirror the traditional construction, in an elementary topos, with appropriate typing. Whereas the generalized membership element $\in_A \hookrightarrow A \times PA$ arises as the pullback of the generic monomorphism $t: 1 \to 2$, along $ev_{A}: A \times 2^{A} \to 2$, NF permits the pullback of $t$ along $ev'_{A}: TA \times 2^{A} \to T2$. Thus for any relation $R \hookrightarrow TA \times B$, obtain the following double pullback:

$$ \xymatrix{R \ar[d] \ar[r] & \in_{A} \ar[d] \ar[r]^{!} & 1 \ar[d]^{t} \\
TA \times B \ar[r]_{1_{TA} \times \hat{r}} & TA \times PA \ar[r]_<<<<{ev'_{A}} & T2} $$

We know, however, that 2 is concretely finite, so strongly cantorian, hence $T2 \cong 2$. Therefore, we conclude that R is uniquely classified by the subobject classifier, as well as the representability of $Sub(TA \times -)$.

$$ \xymatrix{R \ar[d] \ar[r] & \in_{A} \ar[d] \ar[r]^{!} & 1 \ar[d]^{t} \\
TA \times B \ar[r]_{1_{TA} \times \hat{r}} & TA \times PA \ar[r]_<<<<{ev'_{A}} & 2} $$

In the other direction, we construct the stratified analogue of a standard proof (that can be found in \cite{mm}), substituting our definition of pseudo-powerobjects for the classical one. 

$$\xymatrix{TB \times A^B \ar[r]_<<<{1 \times m} \ar[d]_{1 \times !} \ar@/^2pc/[rrr]^{ev'} & TB \times P(A \times B) \ar[r]_<<<<<{v} \ar[d]_{1 \times u} & PA \ar[d]_{\sigma_A} & TA \ar[l]^{\{\cdot\}} \ar[d] \\
TB \times 1 \ar[r]^{1 \times \ulcorner t_B \urcorner} \ar@/_2pc/[rrr]_{!} & TB \times PB \ar[r]^{\in_B} & 2 & 1 \ar[l]_{t}}$$

The unique factorization through the pullback is the pseudo-evaluation morphism. 
\end{proof}

\subsubsection{Pseudo Dependent Products}

A locally cartesian closed category is defined by the existence of left and right adjoints to the pullback functor: $\Sigma_{f} \dashv f^{*} \dashv \Pi_{f}$. As NF is finitely complete, and $\Sigma_{f}$ is just post-composition, $\cat{N}$ has $\Sigma_{f} \dashv f^{*}$, for any $f$. The right adjoint, however, cannot always exist in any consistent model of NF, as it would imply cartesian closure. Nonetheless, we can exhibit for any
pullback functor $f^{*}:\cat N/B\to\cat N/A$ a functor $\tilde{\Pi}_{f}:\cat
N/A\to\cat N/TB$ for which $f^{*}$ is a left adjoint relative to
$\tilde{T}_{B}:\cat N/B\to\cat N/TB$.  Moreover, in the case of $\alpha:A\to1$
(for whatever $A$) we can show that the definition of $\tilde{\Pi}_{\alpha}$
causes $\tilde{\Pi}_{\alpha}\alpha^{*}C$ to coincide with $A\imp C$ analogously
to how the usual composite corresponds to true exponentials.

For $\gamma:D\to A$ (or $(D,\gamma)$ for brevity) in $\cat N/A$,
the object component of $\tilde{\Pi}_{f}(D,\gamma)$ is defined as
\[
\left\{ \tuple{g,\{b\}}|b\in B\,\wedge \, g:(f^{-1}``\{b\})\to
D\,\wedge \,\gamma\circ g=id_{f^{-1}``\{b\}}\right\} \textrm{;}
\]
the map to $TB$ that makes it part of the intended slice category is just the
right projection. One must assign $g$ the same type as its domain,
$f^{-1}``\{b\}$, which must in turn receive the same type as $\{b\}$, which gets
the same type as $B$. It is functorial because any map
$h:(D,\gamma)\to(D',\gamma')$ can be carried to the map $\tilde{h}_{f}$, which
carries $\tuple{g,\{b\}}$ to $\tuple{h\circ g,\{b\}}$, which is well behaved by
virtue of $h$ being a morphism in the slice category.
\begin{prop}
  The pullback functor is a relative left adjoint to
  $\tilde{\Pi}_{f}$.\end{prop}
\begin{proof}
  We prove that for all $(C,\rho)\in\cat N/B$,
  there is a $\tilde{\Pi}_{f}$-universal arrow with domain $(TC,T\rho)$, and
  that this universal arrow is $\tuple{\sigma_{(C,\rho)},f^{*}(C,\rho)}$.  For
  our $\sigma_{(C,\rho)}$ we take the function
  $(TC,T\rho)\to\tilde{\Pi}_{f}f^{*}(C,\rho)$ assigning each $\{c\}\in TC$ to
  the function $\lambda_{a}(\tuple{c,a}):f^{-1}\{\rho(c)\}\to f^{*}(C,\rho)$.
  Given a morphism $m:(TC,T\rho)\to\tilde{\Pi}_{f}(D,\gamma)$, we construct a
  morphism $\acute{m}:f^{*}(C,\rho)\to(D,\gamma)$ by taking $\acute{m}$ to be
  $\lambda_{\tuple{c,a}:C\times_{B}A}\left(m(\{c\})(a)\right)$, where
  $C\times_{B}A$ is the pullback with respect to $f$ and $\rho$.  We then wish
  to show
  \begin{enumerate}
  \item that $\acute{\sigma}_{(C,\rho)}$ is $id_{f^{*}(C,\rho)}$,
  \item that $\tilde{\Pi}_{f}(\acute{m})\circ\sigma_{(C,\rho)}=m$, and
  \item that $\acute{m}$ is the only morphism $f^{*}(C,\rho)\to(D,\gamma)$ that
    satisfies 2.
  \end{enumerate}

  (1) is easy to verify: $\lambda_{\tuple{c,a}}(\sigma(\{c\})(a))$ acts as an identity on tuples. For (2), observe that $\tilde{\Pi}_{f}(\acute{m})$
  just takes each $q\in\tilde{\Pi}_{f}f^{*}(C,\rho)$%
  \footnote{We are here omitting the right element of the pairs that make up the
    $\tilde{\Pi}_{f}$ structure. More strictly, one should be speaking of
    $\tuple{q,f``dom(q)}$, but the second element is a mere formal convenience.
    We will omit the ``tag'' element often for brevity.} to $\acute{m}\circ
    q$. In particular, it takes $\lambda_{a}\left(\tuple{c,a}\right)$ to
    $\lambda_{a}(m(\{c\})(a))$, which is exactly $m(\{c\})$ by extensionality;
    so the function $\{c\}\mapsto\acute{m}\circ\sigma(\{c\})$ is precisely
    $m$. Finally, to prove (3), suppose we had some map
    $m^{\bigstar}:f^{*}(C,\rho)\to(D,\gamma)$ which satisfies (2). In other words, for any $q\in\tilde{\Pi}_{f}f^{*}(C,\rho)$,
    $m^{\bigstar}\circ q=\acute{m}\circ q$.  Since one can show that for any
    $\tuple{c,a}\in C\times_{B}A$ there is a member of
    $\tilde{\Pi}_{f}f^{*}(C,\rho)$ with value $\tuple{c,a}$ for some
    input - namely, $\sigma(\{c\})$ evaluated at $a$ - $m^{\bigstar}$ must agree
    with $\acute{m}$ on all inputs, giving us that $m^{\bigstar}=\acute{m}$.

  With these data, one can establish the more traditional isomorphism:  $\xi:\hom((TC,T\rho),\tilde{\Pi}_{f}(D,\gamma))\cong\hom(f^{*}(C,\rho),(D,\gamma))$
  by taking $\xi(m)=\acute{m}$ and $\xi^{-1}(n)=\tilde{\Pi}_{f}(n)\circ\sigma$.
  Naturality is given by the universal property of $\sigma$.
\end{proof}

The relative adjunction gives NF function spaces, and allows us to derive out earlier (non-local) result, $(A\times-)_{\; T}\negthickspace\dashv(A\imp-)$, as $\Sigma_{\alpha}\alpha_{\; T}^{*}\negthickspace\dashv\tilde{\Pi}_{\alpha}\alpha^{*}$. This turns out to be the corollary to a more general result. 

\begin{lem} 
	For any functors $J,G,F,H$ with $F_{\;
  J}\negthickspace\dashv G$ and $H\dashv F$, we obtain another relative
adjunction $HF_{\; J}\negthickspace\dashv GF$.  
\end{lem}
\begin{proof}
From the relative adjunction we know that
$\hom(Ja,GFb)\cong\hom(Fa,Fb)$, and we know that the ordinary adjunction gives
us $\hom(Ha',b')\cong\hom(a',Fb')$.  Putting $Fa$ in place of $a'$ and $b$ in
place of $b'$, we obtain
\[
\hom(Ja,GFb)\cong\hom(Fa,Fb)\cong\hom(HFa,b),
\]
the desired relative adjunction. 
\end{proof}

The first adjunction, defining local pseudo cartesian closure in $\cat{N}$ can be seen as providing the appropriate notion of unit for pseudo cartesian closure in a given slice category. The appropriate form of (relative) counit is provided by a second relative adjunction.\footnote{For brevity's sake we simply state this result, a full proof of which can be found in the second author's thesis.}

\begin{prop}
	Given any morphism $f: C \to D$\footnote{One might prefer to say ``for any $g: TC \to TD$," but recall $T$ is full and faithful.} the following bijection holds (internally) for any maps $\gamma: A \to TD$ and $\beta: B \to TC$:
	$$ \cat{N}/TC(Tf^*(\gamma),T\beta) \cong \cat{N}/TD(\gamma, \tilde{\Pi}_f(\beta))$$
	In other words, there is a relative adjunction $Tf^*(-) \dashv_{T_C} \tilde{\Pi}_f(\beta)$.	
\end{prop}

The category theoretic content of this pair of relative adjunctions is every slice category of an SPE is itself an SPE. We
refer to this as the \emph{pseudo-fundamental theorem of SPEs}. The set theoretic content of this result relates to NF's handling of dependent sums and products, indexed by an arbitrary set - a longstanding complication. The associated pseudo unit and counit provide the best, most coherent, form (of which the authors are aware) of arbitrary diagonal and projection functions in NF.

\subsection{A Short Word on $T$}

In our definition of an SPE we assumed that $T$ has the nice properties it possesses in $\cat{N}$, when instantiated as $X\mapsto \iota``X$, particularly with respect to limits. It's natural to ask whether $T$ might have an adjoint, or a monad structure. We prove that in $\cat{N}$, $T$ is not part of an adjunction, in any way.

\begin{thm} $T$ is not part of an adjunction, or a composite of an adjoint
  pair.\end{thm}
\begin{proof}\
  \begin{enumerate}
  
  \item \emph{$T$ is not right adjoint}. If $T$ were a right adjoint functor,
    there would be an injection from morphisms $V\to T2$ into morphisms
    $TX_{V}\to T2$, where $X_{V}$ is the object component of the $T$-universal
    arrow for $V$.  However we know that $\left|V\to T2\right|=\left|V\right|$
    while for all $Y$ $\left|TY\to T2\right|\leqslant\left|TV\right|$.
  \item $T$\emph{ is not $G\circ F$ for any $F\dashv G$.} $T$ is an embedding,
    hence a monomorphism, which means that $F$ must also be monic and therefore
    faithful. Any adjunction with a faithful left adjoint must have a monic
    unit; but then there would be an injection
    $V\overset{\eta}{\longrightarrow}GFV=TV$, which is clearly false.
  \item \emph{$T$ is not $F\circ G$ for any $F\dashv G$.} As above, but now $G$
    would be faithful, and the counit would be epimorphic. This would mean a
    surjection $TV=FGV\overset{\epsilon}{\longrightarrow}V$, which is also
    impossible.
  \item \emph{$T$ is not a left adjoint.} Consider the relative adjunction
    \[\hom(TC,A\imp B)\cong\hom(A\times C,B).\] If $T$ has a right adjoint $G$,
    then $\hom(TC,A\imp B)\cong\hom(C,G(A\imp B))$, meaning that $G$ would be a
    functor that took $A\imp B$ and gave us the real categorical exponential
    $B^A$. But there can't always be such an exponential.
  \end{enumerate} \end{proof}

While T fails to be part of an adjunction in NF, nothing obviously prevents it from
being so in a variant of KF+Inf. If the fourth property holds:
\begin{cor} An SPE is a topos if and only if its $T$-functor has a right
  adjoint.\end{cor}

 Another curiosity is the relationship between T and $hom(1,-)$? Both $\cat{N}$ and a well pointed topos are cases of a generalized SPE where T is naturally isomorphic to $hom(1,-)$. In the case of $\cat{N}$, the natural isomorphism provides interesting variants of the Curry-Howard correspondence.  In the following section, the natural isomorphism allows us to form structures we would not expect to, in a stratified theory, suggesting T-functor is natural, rather than a ``syntactic trick.''

\subsection{A Topos Subcategory}

  \begin{thm}
  	The full subcategory of $\cat{N}$ containing the strongly cantorian sets is a topos. Rosser's Axiom of Counting (see \cite{lfm}) would imply this topos has an NNO.
  \end{thm} 
  
  \begin{proof}
  The strongly cantorian sets are closed under taking subsets, since if
$B\subseteq A$, and $A$ is strongly cantorian, then $(\iota\restriction
A)\restriction B$ is also a set; therefore they are closed under
equalisers. They are also closed under products, since for $A\times B$ we can
take the two composites $\iota_A\circ\pi_1$ and $\iota_B\circ\pi_2$ and form the
function $\lambda x.(\tuple{\iota_A\circ\pi_1(x),\iota_B\circ\pi_2(x)})$ as an
instance of ``$\lambda x.(\tuple{f(x),g(x)})$''. This gives us an isomorphism
$A\times B\to TA\times TB$, which we can compose with the isomorphism
$\tuple{\{a\},\{b\}}\mapsto\{\tuple{a,b}\}$ to obtain the singleton function on
$A\times B$. Finally, the strongly cantorian sets are closed under powersets,
since if $A$ is strongly cantorian,
$\lambda x.(\iota_A``x):\sc{A}\to\sc{\iota``A}$ is a set, which can then be
composed with the isomorphism $\iota``A\mapsto\iota(A)$ to give the singleton
function on $\sc{A}$. We have mentioned that $\mathsf{Sub}(TA\times-)$ is
representable by $\sc{A}$; but also $\mathsf{Sub}(TA\times-)$ is isomorphic to
$\mathsf{Sub}(A\times-)$, so that the strongly cantorian sets have power
objects.
\end{proof}

 Everything we rely on for the topos structure of the strongly
cantorian sets holds of the finite sets in KF+Inf, so we may expect a more general result:

\begin{thm}
	Every SPE contains a full subtopos whose objects are the fixed points of the T functor.
\end{thm} 

\section{NF's Internal Category Theory}

To this point, we have developed the category $\cat N$ relative to some other \textit{meta}
theory. In this section we study its role as a set theoretic foundation for category theory.\footnote{A full study of this is carried out in separate work by Vidrine and Lewicki.} We sketch a basic theory of small categories in NF, and prove a Stratified Yoneda Lemma. The final result is a theorem stating the internal category of NF sets is a full internal subcategory. The content of this result is the codomain fibration provides an appropriate notion of a category of sets in (fibred) $\cat{N}$-category theory, despite properly embedding in the externalization of the internal category of NF sets.

The representation of internal diagrams as fibres is motivated by size limitations in classical set theory. In $\cat{N}$, this and the interpretation of ``elements'' as global elements provide a natural motivation for the T-functor.

\subsection{Review of Internal Category Theory}
For a thorough introduction, the reader is
referred to \cite{johnstone}. The basic idea is that the axioms of \naive{}
category theory can be represented diagramatically, generalizing the definition of small categories in $\emph{Set}$ to an arbitrary category with finite limits. 
\begin{mydef}
  Working in some ambient category $\cat{E}$, an internal category $\cat{C}$ is
  a collection of objects and morphisms satisfying the following diagram, with the axioms of category theory expressed as commutative diagrams.
\end{mydef}

\begin{equation}
  \xymatrix{C_1\times_{C_0} C_1 \ar[r]^<<<<<{m} & C_1 \ar@/^/[r]^{d_0} \ar@/_/[r]_{d_1} & C_0 \ar[l]|{i}}
\end{equation}

Intuitively, one thinks of $C_0$ and $C_1$ as the objects of objects and morphisms, respectively; $d_0, d_1, m$ as domain, codomain and composition, and $i$ as the map associating each object of $\cat{C}$ with its identity map. An \emph{internal functor} $\cat{F}$ between internal categories $\cat{C}$ and $\cat{D}$ is a pair of morphisms $F_0:C_0 \to D_0$ and $F_1:C_1 \to D_1$, commuting with the diagrams in the expected way. Natural transformations are defined as maps $C_0 \to D_1$, also satisfying the expected conditions. The resulting 2-category $\cat{cat(E)}$ is nevertheless external for classical set theory. One of the great foundational attributes of NF is $\cat{cat(N)} \in \cat{cat(N)}$.

A presheaf $F$ (from an internal category $\cat{C}$ to the ambient category $\cat{E}$) is represented as a $\gamma_0: F_0 \to C_0$, where the fibre over each element of $C_0$ is its image under $F$.

\begin{mydef}

  Let $\cat{C} \in \cat{cat(E)}$. An \textbf{internal diagram} F on $\cat{C}$ is
  a collection $(F_0, \gamma_0, e)$ such that:

  (i) $\gamma_0:F_0 \to C_0$ $\in \cat{E}/C_0$

  (ii) $e: F_1=F_0\times_{d_0}C_1 \to F_0$

  with the following commutatitivity conditions: $$\gamma_0 \circ e = d_1 \circ
  \pi_2$$ $$e(1\times i) = 1_{F_0}$$ $$e(e\times 1)= e(1\times m): F_2 =
  F_0\times_{C_0} C_1 \times_{C_0} C_1 \to F_0.$$

\end{mydef}

It falls out of the definition of an internal diagram that $F$ itself may be
viewed as an internal category (externally, its category of points) with $F_n (n \geq 1) = F_0 \times_{C_0} C_n$, and
$d_0, d_1$ given as $\pi_1$ and $e$, respectively. Furthermore, we may define a
functor between this new internal category $\cat{F}$ and $\cat{C}$, with
$\gamma_n = \pi_2: F_n \to C_n$ for $n \geq 1$. Hence, we may view $\cat{E^C}$ as a full subcategory of
$\cat{cat(E)/C}$.

As $\cat{N}$ has finite limits, the existence of $\cat{cat(N)}$ is obvious. Futher, the forgetful functor $\cat{cat(N)} \to \cat{N \times N}$ creates finite limits. KF, meanwhile, requires a Dedekind-infinite set to support a robust theory of categories.

\subsection{Representables and Yoneda in NF}

 In NF, $\cat{cat(N)}$ is an internal category, so the standard implementation of internal presheaves may seem unnecessary. It turns out to be far more coherent than the more intuitive definition, suggested by the set theory of NF. In this context, we are able to clarify the typing challenges of representable presheaves and the Yoneda Lemma, and develop a deeper understanding of the nature of the $T$ functor and global elements.

\begin{thm}[\cite{johnstone}]
  Let $\cat{C} \in \cat{cat(E)}$. Then there is an adjunction $R \dashv U$
  $$U:\cat{E^C} \to \cat{E}/C_0; F \mapsto (\gamma_0: F_0 \to C_0)$$ $$
  R:\cat{E}/C_0 \to \cat{E^C} $$ R mapping $\gamma: X \to C_0$ in $\cat{E}/C_0$
  to the discrete opfibration $R(\gamma)$ below.

$$\xymatrix{X\times_{C_0} C_1\times_{C_0} C_1 \ar@/^/[r]^>>>{\pi_{1,2}} \ar@/_/[r]_>>>{1\times m} \ar[d]_{\pi_3} & X \times_{C_0} C_1 \ar[d]^{d_1 \pi_2} \\ 
C_1 \ar@/^/[r]^{d_0} \ar@/_/[r]_{d_1} & C_0}$$

\end{thm}

Diagrams in the image of R are said to be
\textit{representable} - as we might expect, this borrows from intuition in
$\cat{Set}$. Taking $\cat{C}$ to be an internal (i.e. small) category, one can
regard an object U of $\cat{C}$ as a morphism $u:1 \to C_0$, and $R(u)$
corresponds to the covariant representable functor: $\cat{C}(U,-):\cat{C} \to
\cat{Set}$. Specific homsets are defined by pullback, and we obtain the following lemma, useful for working in a stratified theory.

\begin{lem}
  For $\cat{C} \in \cat{cat(N)}$. $Hom(\cat{C}) :=
  \{\cat{C}(U,V)| U, V \in C_0\}$ is the same size as a set of singletons.
\end{lem}

The lemma is an easy consequence of the definition of $Hom(\cat{C})$ as the collection of fibres for the morphism: $\langle d_0, d_1 \rangle : C_1 \to C_0 \times C_0^{op}$.

The counit of $R \dashv U$ gives the internal version of the Yoneda Lemma. The classical form: $\cat{Set^C}(\cat{C}(U,-),F) \cong F(U)$ can be interpreted
internally in $\cat{N}$ as $\cat{N^C}(R(u),F) \cong
\cat{N}/C_0(u,\gamma_0)$. In NF, the departure from
the classical statement, as with Cantor's Theorem, is a direct result of $\cat{N}(1,-)$ providing a \emph{proper} embedding of $\cat{N}$ into itself.

\begin{lem}[NF-Yoneda]
  Given some NF-small category $\cat{C}$ and some covariant presheaf $F:\cat{C}
  \to \cat{N}$, $$Nat(\cat{C}(U,-),F) \cong \cat{N}(1,F(U)) \cong T(F(U)).$$ Furthermore, the
  set of natural transformations is the same size as a set of singletons.
  \footnote{Notice, $Can(F(U)) \to Nat(\cat{C}(U,-),F) \cong F(U)$,
    the standard Yoneda Lemma.}
\end{lem}

\begin{proof}
When we write $\cat{C}(U,-)$ we are referring to the discrete opfibration $R(u)$, in the natural isomorphism, $\cat{N^C}(R(u),F) \cong
\cat{N}/C_0(u,\gamma_0)$. We define F(U) as the following pullback:

$$\xymatrix{ F(U) \ar[r] \ar[d] & F_0 \ar[d]^{\gamma_0} \\
1 \ar[r]_{u} & C_0}$$

We have a bijection between the following hom-sets $\cat{N}/C_0(u,\gamma_0)$ and $\cat{N}(1,F(U))$ by unique factorization of any such $1 \to F_0$ through the pullback defining $F(U)$. The bijection $\cat{N}(1,F(U)) \cong T(F(U))$ is obvious.
\end{proof}

The Yoneda Lemma is a direct generalisation Cayley's Theorem for group representation.\footnote{See any introductory textbook on category theory, or consider the Yoneda Lemma above, in the case where $C_0$ is the terminal object (singleton set) and there is a twist isomorphism on $C_1$ expressing that each morphism is invertible.} The concept of representation leads us to consider an uncomfortable side effect of set theoretic foundations: their semantics are dependent on decisions we make, and rules we set, for implementation of mathematical objects. NF has some fairly robust literature on this subject: ordered pairs, ordinals and cardinals being the obvious examples. For category theory, the representation of indexed families as fibres is a key challenge - explicitly, the restriction of indexing sets to $\iota``V$.\footnote{There is a sense in which ordinals seem like they should just be indexed families with a well-ordering on the indexing set. In unpublished work, the authors have actually developed notions of ``Quine Sequences'' of ordinal size that get NF closer to this - but, of course, a consistent theory with a set of all ordinals can only get ``close.''} 

\subsubsection{A Full Internal Subcategory}
The size restriction placed upon us by implementing representable functors as fibres extends to the codomain fibration, $\cat{cod}:\cat{N}/\cat{N} \to \cat{N}$, where $\cat{N}/\cat{N}$ is the ``arrow category'' of $\cat{N}$, with morphisms as objects, and commutative squares as morphisms. Intuitively, $\cat{cod}$ should represent $\cat{N}$-indexed families of NF sets, but we are given pause. The externalization of the internal category of NF-Sets contains families indexed by all sets of NF, rather than restricting to those of cardinality no greater than $\iota``V$. Nevertheless, from $cod(\cat{N})$, we can generate an internal full subcategory that is precisely the internal category of NF sets. The full proof of this result, and its relevance to the category theory of NF is outside of the scope of this paper. In particular, an internal category of NF sets in $\cat{N}$ being full relates to the ``modern'' idea of a Grothendieck Universe as a full internal topos, within a topos. We state our main theorem, and refer the reader to collaborative work of the second and third author, likely to first appear in the former's thesis. $\in_{\cat{N}}$ refers to the stratified set membership relation $\{ \langle \{x\}, y \rangle | x \in y \} \subset TV \times PV = TV \times V$.

\begin{thm}
There is an internal full subcategory of NF Sets, generated by $\pi_{2}' := \Gamma:\in_{\cat{N}} \subset TV \times V \to V$, viewed as an element of the fibre $\cat{cod}^{-1}(V)$ over $V$. The associated exponential in $\cat{N}/ V \times V$ is precisely the set of NF functions, with associated domain and codomain maps: $\langle dom, cod \rangle: Funct \to V \times V$. The generic morphism, $ev: dom^{*}\Gamma \to cod^{*}\Gamma$ is defined by the action: $\langle f, \{x\} \rangle \mapsto \langle f, \{f(x)\} \rangle$.	
\end{thm}

\section{NF As A Category of Classes}
The category of NF Sets has the structure of a category of classes,
in the sense of \cite{awodey, joyal}. This motivates the question: does the set folklore definition of smallness, strongly cantorian, hold up in category theory. Following Joyal and Moerdijk, we consider fibrewise smallness: \textit{strongly cantorian maps}. This requires us to extend NF to the theory NF + SCU. In this section we prove some basic results about $\cat{N}$ as a category of classes, and examine some interesting properties of NF + SCU.

\subsection{Class Categories and NF}
For a proper introduction to Algebraic Set Theory, the reader is referred to works of Joyal \& Moerdijk, and
Awodey.\cite{awodey, joyal}  The idea
of class categories is to form an internal ``algebraic" model of a set theory, using a universal object and defining a system of small maps. A natural example is the ideal completion of a topos. The universe in a class category can exist in varying strengths - allowing the category of ``sets'' to be defined as a full subcategory or, as is true in the NF case, an internal category. 

A \textit{Class Category} $\cat{C}$ is a category containing the
following four properties:

(i) A Heyting category $\cat{C}$ of classes.

(ii) A subcategory $\cat{S} \subset \cat{C}$ of small objects (``sets'')

(iii) A powerclass functor $P:\cat{C} \to \cat{C}$ (``restricted''
powerobjects)

(iv) A universe $V$ with, at least, a monomorphism $PV \to V$.

\begin{mydef}

  A \textbf{category of classes} $\cat{C}$ is a category with the following
  conditions:

  (i) $\cat{C}$ has finite limits and finite coproducts

  (ii) $\cat{C}$ has kernel quotients, and regular epis are stable under pullback.

  (iii) $\cat{C}$ has dual images, that is, for every arrow $f:C \to D$, the
  pullback functor $f^*:Sub(D) \to Sub(C)$ has a right adjoint, $f_*:Sub(D) \to
  Sub(C)$.

\end{mydef}

Conditions i and ii imply that $f^*$ also has a left adjoint $f_!$,
with $f_! \dashv f^* \dashv f_*$ satisfying the Beck-Chevalley conditions. This permits one to model first order logic with equality in $\cat{N}$. 

\begin{thm}
  $\cat{N}$, the category of NF Sets, is a category of classes.
\end{thm}

\begin{proof}
  Conditions i \& ii have already been proven; iii simply requires observing that for any $f: A \to B$, the primitive notation defining $\forall_f$ is clearly stratified: $(T \in Sub(A)) \mapsto \{b| \forall a. f(a)=b \Rightarrow a \in T\}$.
\end{proof}

Notice, viewed as a partial order (i.e. its subobject language), $\cat{N}$ has full (constructive) quantification. Typing issues arise when one moves to the type theoretic quantification (dependent sums and products) considered in section 3.

\begin{mydef}
  Let $\cat{C}$ be a category of classes, we define a \textbf{system of
    small maps on $\cat{C}$} as a collection of arrows $\cat{S}$
  satisfying the following:

  (i) $\cat{S} \subset \cat{C}$ is a subcategory of $\cat{C}$, with
  $Ob(\cat{S}) = Ob(\cat{C})$

  (ii) The pullback of a small map along any map is small.
  
  (iii) Diagonals $\Delta: C \to C\times C$ are small

  (iv) If $f \circ e$ is small and $e$ is a regular epimorphism, then $f$ is
  small.

  (v) Copairs of small maps are small.

\end{mydef}

\begin{myax}[SCU]
  The sumset of a strongly cantorian set of strongly cantorian sets is strongly
  cantorian.
\end{myax}

\begin{thm}
  In NF + SCU, the strongly cantorian
  maps (those with strongly cantorian fibres), labeled $\cat{SC}$, form a system
  of small maps.
\end{thm}

\begin{proof}

  In the presence of SCU, it is trivial that $\cat{SC}$ forms a subcategory. For any pullback $g^{*}(f)$, with $f \in \cat{SC}$, and any fibre over $\pi_1^{*}(a')$, $a' \in cod(g)$, there is an injection of this  fibre into the fibre of $f$ over the image $g(a') \in cod(f)$. As the latter is strongly cantorian and $|f'^{-1}(a')| \prec |f^{-1}(b)|$, the fibre over an arbitrary $a'$ is strongly cantorian, so $g^{*}(f) \in \cat{SC}$. iii \& v are straightforward, so all that remains is iv. Following the same line of reasoning, one can show that, as
  $e$ is surjective, if the fibre of $f \circ e$ over any element is a strongly cantorian set, then the fibre of $f$ over that element must be strongly cantorian.
 
\end{proof}

\begin{cor}[Descent Property for $\cat{SC}$]

  If $f$ is strongly cantorian and $e$ a regular epi, fitting
  into a pullback diagram below, then $g$ is strongly cantorian. $$\xymatrix{D \ar[r] \ar[d]_f & B \ar[d]^g \\
    C \ar[r]_e & A}$$

\end{cor}

The third aspect of a Class Category is that it contain \emph{powerclasses}.

\begin{mydef}
  For every object $C$ in a class category $\cat{C}$, a \textbf{powerclass} is
  an object \textbf{P}C with a small relation $\epsilon_C \to C\times
  \textbf{P}C$ such that, for any $X$ and any small relation $R \to C\times X$,
  there is a unique arrow $\rho$ satisfying the pullback diagram below:
$$ \xymatrix{ R \ar[r] \ar[d] & \epsilon_C \ar[d] \\
  C\times X \ar[r]_{1_C \times \rho} & C\times \textbf{P}C}$$ In addition,
$\textbf{P}C$ satisfies the condition that the internal subset relation
$\subset_C \to \textbf{P}C \times \textbf{P}C$ is small.

\end{mydef}

This might not seem problematic in $\cat{N}$, but result would be the Burali-Forti paradox.

\begin{prop}
	NF + SCU, with the small maps defined as those with strongly cantorian fibres, cannot form a powerclass functor.
\end{prop}

\begin{proof}
Consider the intersection of the set of strongly cantorian sets, $P_{S}(V)$, with the set of all wellorderings. Such a set would be closed
under the formation of binary products, and therefore under the formation of
orderings. For a wellordering relation $\leq$, the induced
relation on singletons $\leq^\iota$ has the same order type. While the following is not true generally in NF, for a \emph{strongly cantorian} ordinal, it can
be shown by transfinite induction that every ordinal number (i.e. equivalence class of wellorderings under order isomorphism) is the order type of all the ordinals beneath it. But then the collection of ordinals formed from the strongly cantorian wellorders will itself be a strongly cantorian wellorder, longer than any in the set of strongly cantorian wellorderings - precisely the contradiction of Burali-Forti.
\end{proof}\label{bfort}

\subsection{SCU as an Axiom}

\begin{thm}
	SCU is a theorem of NFU + Choice
\end{thm}

\begin{proof}
Let $X$ be a
strongly cantorian set of strongly cantorian sets.  AC implies that every
strongly cantorian set is the same size as an initial segment of the ordinals
(and all the ordinals in that inital segment will be cantorian).  Use AC to pick
one such bijection for each $x \in X$ and fix such a bijection for $X$ itself.
Thus everything in $\bigcup X$ has an address that is an ordered pair of
cantorian ordinals, so $\bigcup X$ now injects into a set of ordered pairs of
cantorian ordinals.  Any such set is strongly cantorian, so $\bigcup X$ must be
strongly cantorian too. 

\end{proof}

SCU doesn't appear to be a theorem of NF, but nor does
it appear to be strong. One might hope to prove its relative consistency by means of Rieger-Bernays permutation
models, but it turns out that SCU is invariant.\footnote{We omit the proof of the following lemma, as it is intended more for advanced students of NF and finds no further need in this paper. The relevant background on Rieger-Bernays can be found in \cite{forster2}.}
\begin{lem}
  SCU is invariant 
\end{lem}

We start with a banal observation.  Let $F_1$ be the function that sends each
strongly cantorian set $x$ to $\iota \restric x$. $F_1$ cannot be a set: if it
were then $F_1``(\iota``V) = \{\iota \restric \{x\}: x \in V\}$ would be a set
(beco's the image of a set in a set is a set) and $\bigcup F_1``(\iota``V)$
would be the graph of the singleton function, and that cannot be a set.  However
this line of talk leaves open the possibility that $F_1 \restric x$ might be a
set whenever $x$ is strongly cantorian. In fact we have the following.

\begin{myres}
  SCU is equivalent to the assertion that, for all strongly cantorian sets $x$
  of strongly cantorian sets, $F_1\restric x$ is a set.
\end{myres}

\begin{proof}

\smallskip

L $\to$ R

Assume SCU and let $X$ be a strongly cantorian set of strongly cantorian sets.
Then $\iota\restric \bigcup X$ is a set.  Let's call it $F$.  Consider now the
function that sends each $x \in X$ to $F\restric x$.  This is a set, since it is
the extension of a stratifiable set abstract.  But $X$ was an arbitrary strongly
cantorian set of strongly cantorian sets.  So SCU implies that $F_1$ is locally
a set, in the sense that, for any strongly cantorian set $X$ [the graph of] its
restriction to $X$ is a set.

\smallskip

R $\to$ L

Let $X$ be a strongly cantorian set of strongly cantorian sets.  Then $F_1
\restric X = \lambda x \in X.\iota\restric x$ is a set and so too is the image
of $X$ in it, namely $\{\iota\restric x: x \in X\}$.  But then $\bigcup
\{\iota\restric x: x \in X\}$ is a set, and is $\iota\restric \bigcup X$ making
$\bigcup X$ strongly cantorian as desired.

\end{proof}

Consider now
the function $F_2: X \mapsto F_1 \restric X$ for every strongly cantorian set
$X$ of strongly cantorian sets.  Can the graph of $F_2$ be a set?  Clearly not:
$\iota^2``V$ is a set of strongly cantorian sets of strongly cantorian sets, and its
image in this function would be the set $\{\iota\restric\{x\}: \{x\} \in
\iota``V\}$, which is $\{\iota\restric \{x\}: x \in V\}$, whose sumset is simply
the graph of $\iota$.  However, there seems to be no obvious objection to the
existence of [the graph of] the restriction of $F_2$ to any {\sl strongly
  cantorian} set.

Let us write `stcan' for the class of strongly
cantorian sets, `stcan$^2$' for the class of strongly cantorian sets of strongly
cantorian sets.  Let $F_n$ be the function $\lambda x \in$ stcan$^n.F_n\restric x$; $F_n$ cannot exist globally but $F_n \restric X$ can exist for any $X$ in stcan$^{n+1}$.

Let SCU$_n$ be the assertion that restrictions of $F_n$ exist locally, so that $F_n \restric X$ is a set whenever $X \in$ stcan$^n$. SCU$_1$ is of course SCU.

We record for later use the trivial lemma:

\begin{lem} 
	SCU implies that if $x \in$ stcan$^{n+1}$ then $\bigcup x \in$ stcan$^n$.
\end{lem}

\begin{thm} 
	All SCU$_n$ for $n \in \Nn$ are equivalent.
\end{thm}

\begin{proof}

SCU$_{n+1}$ implies SCU$_n$.

\medskip

Suppose $x \in$ stcan$^n$; we will show that $F_n\restric x$ exists.  Since $x
\in$ stcan$^n$ we have $\iota``x \in$ stcan$^{n+1}$.  So, by SCU$_{n+1}$,
$F_{n+1}\restric \iota``x$ exists. This is the function that, on being given
$\{y\} \in \iota``x$, returns $F_n\restric\{y\}$.  This value is the singleton
$\{\tuple{y, F_n(y)}\}$.  So $F_{n+1}``(\iota``x)$ (which is a set) is
$\{\{\tuple{y,F_n(y)}\}:y\in x\}$, and the sumset of this last object is
precisely $F_n\restric x$, as desired.

% Let's start with a concrete case.  Suppose $x \in$ stcan$^3$ (so that $x$ is a
% strongly cantorian set of strongly cantorian sets of strongly cantorian sets).
% Then, by SCU, $\bigcup X$ is a strongly cantorian set (of strongly cantorian
% sets) so---by SCU again---the restriction to $\bigcup X$ of the function $x
% \mapsto \iota\restric x$ is a set. Let's call this function $H$ for the
% moment.  But then the function that takes subsets $S$ of $\bigcup X$ and
% returns the restriction $H \restric S$ is also a set.  So the restriction of
% this function to $x$ is a set.

\medskip

For the other direction we assume SCU$_n$, and suppose $x$ to be an arbitrary
member of stcan$^{n+1}$; we will show that $F_{n+1}\restric x$ is a set.

Clearly stcan$^n \subseteq$ stcan$^{n+1}$ so $x \in$ stcan$^n$, whence---by
SCU---$\bigcup x \in$ stcan$^n$.  SCU$_n$ now tells us that $F_n \restric
\bigcup x$ is a set. Let's call this function $H$ for the moment.  But then the
function that takes subsets $S$ of $\bigcup x$ and returns the restriction $H
\restric S$ is also a set.  $H$ is defined on $\mathcal{P}(\bigcup x)$ which is
a superset of $x$. So the restriction of this function to $x$ is a set.

\end{proof}

SCU implies that the product of a strongly cantorian family of strongly cantorian sets is strongly cantorian.

\begin{thm}\label{SCUproduct} (SCU)
	
	For all $I$, if stcan$(I)$ and $(\forall i \in
	I)($stcan$(A_i))$ then stcan($\displaystyle{\prod_{i \in
			I}A_i}$)
	
	\end{thm}

\begin{proof}

The product is a subset of $\mathcal{P}(\displaystyle{
	\bigcup_{i \in I} A_i \times I})$.  Assuming SCU the union
$\displaystyle{\bigcup_{i \in I}A_i}$ is strongly cantorian because
$I$ is and all the $A_i$ are.  The cartesian product of two strongly
cantorian sets is strongly cantorian, a power set of a strongly
cantorian set is strongly cantorian, and every subset of a strongly
cantorian set is strongly cantorian.  

\end{proof}

\medskip

We can now prove
\begin{thm} (SCU)
	
	Let $\tuple{I, \leq_I}$ be a directed poset with $I$ strongly
	cantorian, and let $\{A_i: i\in I\}$ be a family of sets with
	surjections $\pi_{i,j}:A_i \twoheadrightarrow A_j$ whenever $i >_I j$, and the
	surjections all commute.  Suppose further that, for every $i$ and
	$j$, the fibres of $\pi_{i,j}$ are strongly cantorian.  Naturally
	there is a limit object $A_I$, a least thing that maps onto all the
	$A_i$---with maps $\pi_{I,i}: A_I \twoheadrightarrow A_i$ for each $i \in I$.
	
	Then all the fibres of $f_{I,i}$ are strongly cantorian.\end{thm}

%24/xi/2015

\begin{proof}

The inverse (projective) limit $A_I$ is $$\{f \in \prod_{i \in I}: (\forall j >_I i)(\pi_{j,i}(f(j)) = f(i)\}$$

For $x \in A_i$, the fibre $\pi_{I,i}^{-1}``\{x\}$ is 

$$\{f \in \prod_{j > i \in I}: (\forall j >_I i)(\pi_{j,i}(f(j)) = x)\}$$

So a fibre for $x \in A_i$ is set of functions $f$ that, for each $j >
i \in I$, pick something that $\pi_{j,i}$ sends to $x$.  So it's a
subset of the product of all the subsets $\pi_{j,i}^{-1}``\{x\}$ of
$A_j$ \ldots and, by assumption, all those sets are strongly
cantorian.  So the fibre is a subset of a direct product of a strongly
cantorian family of strongly cantorian sets and accordingly, by
theorem \ref{SCUproduct}, is strongly cantorian.

\end{proof}

In plain language, SCU implies that the inverse limit of
a strongly cantorian family of strongly cantorian structures is
strongly cantorian.

\medskip

Our invocation of SCU was intended to allow composition of small functional relations. Clearly the same holds true for relational composition of small relations (defined where each set $aR-$ and $-Rb$ is strongly cantorian). What holds, however, is something quite a bit stronger, with relevance to many aspects of category theory in NF.

In order to work with infinite strongly cantorian families, we assume the of counting in the following theorems.

\begin{lem}\label{smallgraphs}(SCU)
	
	If $G$ is a connected graph wherein, for every element $x$, the set
	$N(x)$ of neighbours of $x$ is strongly cantorian, then the edge set
	and the vertex set of $G$ are both strongly cantorian.\end{lem}

\begin{proof}

Fix a vertex $v$ and consider the sequence $\tuple{N_n(v): n \in \Nn}$
where $N_n(v)$ is the set of vertices distant at most $n$ from $v$.
Naturally we expect to be able to prove by induction that
stcan$(N_n(v))$ but of course this is not possible.  What we {\sl can}
do, however, is prove by induction on `$n$' that $\iota \cap
(N_{Tn}(v) \times \iota``N_n(v))$ exists.  (This is a weakly
stratified induction).  Observe that it is true for $n = 1$.  Suppose
true for $n$, which is to say that $\iota \cap (N_{Tn}(v) \times
\iota``N_n(v))$ exists.  But, by the axiom of counting, $N_{Tn}(v) =
N_n(v)$, so the restriction of the singleton function whose existence
we have assumed is $\iota \restric N_n(v))$.  But now we can use SCU
in the induction step in the way we always intended, and conclude that
$\iota \restric N_{n+1}(v)$ exists.  But this is to say that $\iota
\cap (N_{Tn+1}(v) \times \iota``N_{n+1}(v))$ exists, and the weakly
stratified induction step is concluded.  Thus, for all $n$, $\iota
\cap (N_{Tn}(v) \times \iota``N_n(v))$ exists.  But, by the axiom of
counting, this is to say $\iota \cap (N_{n}(v) \times \iota``N_n(v))$
exists for all $n$ \ldots in other words $\iota \restric N_n(v)$
exists for all $n$.

But now (the vertex set of) $G$ is a union of a strongly cantorian
family of strongly cantorian sets and is strongly cantorian by SCU.

\end{proof}

We already know strongly cantorian sets support coequalisers of morphisms between them, but we can now strengthen this result.

\begin{thm} \label{Small maps have coequalisers}
	
	SCU implies coequalizers exist for any small maps, $f,g:A \to B$.\end{thm}

\begin{proof} 

We consider the 
quotient of $B$ under the following equivalence relation on $B$:
$t(\{\tuple{b_b,b_2}|(\exists a)(f(a) = b_1 \wedge g(a) = b_2)\})$.

We want this quotient to be strongly cantorian.  Each element of
the quotient can be thought of as a graph---and, indeed, as a
connected graph---wherein, for every element $x$, the set $N(x)$ of
neighbours of $x$ is strongly cantorian.  We then invoke lemma
\ref{smallgraphs} to conclude that each equivalence class is small,
but that then means that the quotient map is small.
\end{proof}

\end{document}